\documentclass[12pt,a4paper]{article}
\usepackage{fullpage, amsfonts, amsthm, amsmath, setspace, graphicx, amssymb, float, colonequals, hyperref, float,mathrsfs,bookmark}
\usepackage[usenames,dvipsnames]{xcolor}

\newtheorem{conjecture}{Conjecture}[section]
\newtheorem{theorem}{Theorem}[section]
\newtheorem{lemma}[theorem]{Lemma}

\newtheorem{corollary}[theorem]{Corollary}

\numberwithin{subcase}{case}
\numberwithin{subsubcase}{subcase}
\numberwithin{claim}{theorem}

\newenvironment{note}[1][Note]{\begin{trivlist}
\item[\hskip \labelsep {\bfseries #1}]}{\end{trivlist}}

\newenvironment{corollary-mannum}[2][Corollary]{\begin{trivlist}
\item[\hskip \labelsep {\bfseries #1}\hskip \labelsep {\bfseries #2}]}{\end{trivlist}}
\newenvironment{conjecture-mannum}[2][Conjecture]{\begin{trivlist}
\item[\hskip \labelsep {\bfseries #1}\hskip \labelsep {\bfseries #2}]}{\end{trivlist}}

\newcommand{\ds}{\displaystyle}

\title{\bf On approximate Gauss--Lucas theorems.}

\author{Trevor Richards\thanks{Email: trichards@monmouthcollege.edu}\vspace{6mm}\\{\em Department of Mathematics and Computer Science, Monmouth College}\\{\em Monmouth, IL United States}}

\begin{document}

\maketitle
\begin{abstract}

The Gauss--Lucas theorem states that any convex set $K\subset\mathbb{C}$ which contains all $n$ zeros of a degree $n$ polynomial $p\in\mathbb{C}[z]$ must also contain all $n-1$ critical points of $p$.  In this paper we explore the following question: for which choices of positive integers $n$ and $k$, and positive real number $\epsilon$, will it follow that for every degree $n$ polynomial $p$ with at least $k$ zeros lying in $K$, $p$ will have at least $k-1$ critical points lying in the $\epsilon$-neighborhood of $K$.  We supply an inequality relating $n$, $k$, and $\epsilon$ which, when satisfied, guarantees a positive answer to the above question.

\end{abstract}

\begin{center}


\end{center}
\begin{center}

\end{center}
\section{Overview and definitions.}%

The Gauss--Lucas theorem states that if all $n$ of the zeros of a degree $n$ complex polynomial $p$ lie in a convex set $K\subset\mathbb{C}$, then all $n-1$ of the critical point of $p$ (that is, the zeros of $p'$) lie in $K$ as well.  One might expect that if at least $k$ zeros of a degree $n$ polynomial $p$ lie in $K$, then at least $k-1$ of the critical points of $p$ will lie in $K$.  In general this fails badly.  Consider the example (drawn from~\cite{T}) where $K=[0,1]$, and $p$ is a polynomial with $n-1$ distinct zeros in $K$, and a single zero at $i$.  A little arithmetic shows that in fact every critical point of $p$ will have strictly positive imaginary part, and thus will lie outside of $K$.  However for any fixed neighborhood $\mathcal{O}$ of $K$, if $n$ is sufficiently large, then all $n-1$ of the critical points of $p$ will lie in $\mathcal{O}$ (as will follow from Theorem~\ref{thm: Main inequalities.} to come).  This last observation suggests the following property, which we will explore in this paper.  For positive integers $n$ and $k$, with $k\leq n$, and for $\epsilon>0$, we say that $K$ satisfies the approximate Gauss--Lucas property (denoted $\mathcal{AGL}(n,k,\epsilon)$) if:

\medskip

\begin{minipage}{.9\textwidth}
\begin{note}[$\mathcal{AGL}(n,k,\epsilon)$:]
Every degree $n$ polynomial with at least $k$ zeros lying in $K$ has at least $k-1$ critical points lying in the $\epsilon$-neighborhood of $K$.
\end{note}
\end{minipage}

\medskip

The stronger property, where ``polynomial'' is replaced by ``rational function'', and $n$ denotes not the degree of the rational function, but the combined number of zeros and poles of the rational function, will be denoted $\mathcal{RAGL}(n,k,\epsilon)$ (the ``rational approximate Gauss--Lucas property'').  For a bounded convex set $K\subset\mathbb{C}$, our goal is to find conditions on $n$, $k$, and $\epsilon$ which will ensure that $K$ satisfies $\mathcal{RAGL}(n,k,\epsilon)$.  Our main theorem is the following.

\begin{theorem}\label{thm: Main inequalities.}
Let $K\subset\mathbb{C}$ be bounded and convex, with diameter $s\geq0$.  For positive integers $n$ and $k$, with $k\leq n$, and a positive real number $\epsilon$, if $n$, $k$, and $\epsilon$ satisfy the inequality $$\dfrac{16(s+\epsilon)^2}{\epsilon^2}<\dfrac{k}{(n-k)^2},$$ then $K$ satisfies $\mathcal{RAGL}(n,k,\epsilon)$.

Moreover, if $K$ is known to be a disk, it suffices for $n$, $k$, and $\epsilon$ to satisfy $$\dfrac{8(s+\epsilon)}{\epsilon}<\dfrac{k}{(n-k)^2}$$ to ensure that $K$ satisfies $\mathcal{RAGL}(n,k,\epsilon)$.
\end{theorem}

Historically, there have been two different approaches to determining for which values of $n$, $k$, and $\epsilon$ the set $K$ will satisfy $\mathcal{RAGL}(n,k,\epsilon)$.  The first, pioneered by S.~Kakeya~\cite{K} views $n$ and $k$ as fixed, and seeks the smallest possible $\epsilon$ for which $K$ satisfies $\mathcal{RAGL}(n,k,\epsilon)$.  The second, which views $\epsilon$ as fixed, and seeks to find choices of $n$ and $k$ for which $K$ satisfies $\mathcal{RAGL}(n,k,\epsilon)$, is related to the~2016 asymptotic Gauss--Lucas theorem of V. Totik~\cite{T}.  Let us first introduce some notation.  For a set $G\subset\mathbb{C}$ and a rational function $f\in\mathbb{C}(z)$, let $\#_z(f,G)$, $\#_p(f,G)$, and $\#_c(f,G)$ denote the number of zeros, poles, and critical points of $f$ respectively which lie in $G$.  For $\epsilon>0$, let $G_\epsilon$ denote the $\epsilon$-neighborhood of $G$.  Thus the $\mathcal{RAGL}(n,k,\epsilon)$ property may be restated as: 

\medskip

\begin{minipage}{.9\textwidth}
\begin{note}[$\mathcal{RAGL}(n,k,\epsilon)$:]
For every rational function $f$ with $\#_z(f,\mathbb{C})+\#_p(f,\mathbb{C})=n$, if $\#_z(f,K)\geq k$ then $\#_c(f,K_\epsilon)\geq k-1$.
\end{note}
\end{minipage}

\medskip

\subsection{Fixing \texorpdfstring{$n$}{TEXT} and \texorpdfstring{$k$}{TEXT}, and seeking \texorpdfstring{$\epsilon$}{TEXT}.}

In 1917, S. Kakeya~\cite{K} showed that if $K$ is the disk $D_R$ centered at the origin with radius $R\geq0$, then there is a function $\epsilon=\psi_R(n,k)<\infty$ such that $D_R$ satisfies $\mathcal{AGL}(n,k,\psi_R(n,k))$, and he showed moreover that for any $n$ and $k$, $\psi_R(n,k)=R\cdot\psi_1(n,k)$.  In this work Kakeya was also able to find $\psi_1(n,k)$ in the special case where $k=2$: $\psi_1(n,2)=\csc(\pi/n)-1$.  (Note: in~\cite{K} and elsewhere in the literature, the quantity sought, rather than $\epsilon$, is $R'=R+\epsilon$, the radius of the larger circle guaranteed to contain the expected number of critical points.  In this paper, we focus on the ``$\epsilon$-cushion'' needed around the set $K$, as this concept still makes sense in the context of bounded convex sets which are not disks.)

Two upper bounds are known for $\psi_1(n,k)$ in general.  First in~1939, M. Marden~\cite{M} established that \begin{equation}\label{eqn: Marden upper bound on psi.}
\psi_1(n,k)\leq\csc\left(\dfrac{\pi}{2(n-k+1)}\right)-1.
\end{equation}  Second, in 1945, M. Biernacki~\cite{B} established that \begin{equation}\label{eqn: Biernacki upper bound on psi.}
\psi_1(n,k)\leq\ds\left(\prod_{j=1}^{n-k}\dfrac{n+j}{n-j}\right)-1.
\end{equation}

Neither of the bounds of Marden and Biernacki is universally better than the other, however since Marden's bound is a function of $n-k$, we see that if $n-k$ is held constant, and $n$ is allowed to approach $\infty$, Marden's bound remains constant, while Biernacki's bound approaches zero.  If we solve the first inequality given in Theorem~\ref{thm: Main inequalities.} for $\epsilon$, we obtain an upper bound for the quantity corresponding to $\psi_R$ in the general case where $K$ is bounded and convex, but not necessarily a disk, and $f$ is rational, not necessarily a polynomial.  That is, we obtain the following.

\begin{corollary}\label{cor: n and k held constant, solve for epsilon.}
Let $K\subset\mathbb{C}$ be bounded and convex with diameter $s$, and let $k$ and $n$ be positive integers with $k<n$, and such that $\sqrt{k}>4(n-k)$.  If $$\epsilon>\dfrac{4s(n-k)}{\sqrt{k}-4(n-k)},$$ then $K$ satisfies $\mathcal{RAGL}(n,k,\epsilon)$.
\end{corollary}

In the case where $K$ is a disk, the second inequaltiy in Theorem~\ref{thm: Main inequalities.} gives us the following bound on $\psi_1$.

\begin{corollary}\label{cor: Our bound on psi_1.}
Let $k$ and $n$ be positive integers with $k<n$, and such that $k>8(n-k)^2$.  Then $$\psi_1(n,k)\leq\dfrac{8(n-k)^2}{k-8(n-k)^2}.$$
\end{corollary}

It is difficult to determine in exactly which cases any one of the given three upper bounds on $\psi_1$ (Marden's, Biernacki's, and Corollary~\ref{cor: Our bound on psi_1.} above) is the lowest.  However, as mentioned before, Marden's bound is constant as a function of $n-k$.  Biernacki's bound, on the other hand, may be written as
\begin{equation}
\ds\left(\prod_{i=1}^{n-k}\dfrac{n+i}{n-i}\right)-1=\dfrac{p(1/n)}{n+q(1/n)}
\end{equation}
where $p$ and $q$ are polynomials, whose coefficients are functions of $n-k$, and the constant term of $p$ is $(n-k)(n-k+1)$.  Thus Biernacki's bound implies that if $n-k$ is fixed, and $n\to\infty$, then $\psi_1(n,k)\in O(1/n)$.  This fact follows also from our bound in Corollary~\ref{cor: Our bound on psi_1.}.  Note that since (again when $n-k$ is fixed and $n\to\infty$) Biernacki's bound is on the order of $(n-k)(n-k+1)/n$, while the bound in Corollary~\ref{cor: Our bound on psi_1.} is on the order of $8(n-k)^2/n$, so that Biernacki's bound is superior to Corollary~\ref{cor: Our bound on psi_1.} in the asymptotic sense.

For more information on this approach, see Sections~25 and~26 of~\cite{M2}.

\subsection{Fixing \texorpdfstring{$\epsilon$}{TEXT}, and seeking \texorpdfstring{$n$}{TEXT} and \texorpdfstring{$k$}{TEXT}.}

In 2016, V. Totik~\cite{T} established the following approximate and asymptotic version of the Gauss--Lucas theorem.

\begin{theorem}\label{thm: Asymp. GL.}
Let $K\subset\mathbb{C}$ be bounded and convex, and let $\epsilon>0$ be given.  For any sequence of polynomials $\{p_n\}$, with $\deg(p_n)=n$, if $\dfrac{\#_z(p_n,K)}{n}\to1$, then $\dfrac{\#_c(p_n,K_\epsilon)}{n-1}\to1$.
\end{theorem}

(Note that another proof of Theorem~\ref{thm: Asymp. GL.} is posted on the arXiv~\cite{BKS} by R. Boegvad et. al. using more classical methods of complex analysis, where the proof of V. Totik uses results from logarithmic potential theory.)

We believe that underlying this asymptotic result (that is, Theorem~\ref{thm: Asymp. GL.}) is a static principle, namely that for any $\epsilon>0$, if a sufficiently high fraction of the zeros of a polynomial $p$ lie in $K$, then a similarly high fraction of the critical points of $p$ lie in $K_\epsilon$ (where ``sufficiently high'' depends only on $K$ and $\epsilon$, and not on the degree of $p$).  That is, we make the following conjecture, again including the possibility of poles.

\begin{conjecture}\label{conj: epsilon fixed goal.}
Let $K\subset\mathbb{C}$ be bounded and convex, and let $\epsilon>0$ be given.  There is some constant $C>0$ such that if $k/(n-k)>C$, then $K$ satisfies $\mathcal{RAGL}(n,k,\epsilon)$.
\end{conjecture}

Clearly a positive result for this conjecture would immediately imply the conclusion of Theorem~\ref{thm: Asymp. GL.} as well.  In support of Conjecture~\ref{conj: epsilon fixed goal.}, we observe that the first inequality in Theorem~\ref{thm: Main inequalities.} establishes the result of Conjecture~\ref{conj: epsilon fixed goal.} subject to the stronger condition that the ratio of $k$ to the \textbf{square} of $(n-k)$ is sufficiently large.

\section{A BRIEF LEMMA}\label{sect: Lemma.}%

\begin{lemma}\label{lem: Lemma.}
Let $r\in\mathbb{C}(z)$ be a rational function, all of whose zeros and poles have multiplicity one.  Then the zeros of the logarithmic derivative $\dfrac{r'}{r}$ are exactly the critical points of $r$, with the same multiplicities, and the poles of $\dfrac{r'}{r}$ are the zeros and poles of $r$, each with multiplicity one.
\end{lemma}

\begin{proof}
Let $p,q\in\mathbb{C}[z]$ be polynomials such that $r=\dfrac{p}{q}$ is in lowest form.  A bit of arithmetic gives that $$\dfrac{r'}{r}=\dfrac{qp'-pq'}{pq}.$$  The numerator $qp'-pq'$ is exactly the numerator of $r'$, and since each zero and pole of $r$ has multiplicity one, no zero of either $p$ or $q$ is a zero of $qp'-pq'$, so that no terms cancel in our expansion of $\dfrac{r'}{r}$.  Moreover, the poles of $\dfrac{r'}{r}$ are just the zeros of $p$ and $q$, namely the zeros and poles of $r$.  This completes the proof.
\end{proof}

\section{PROOFS}\label{sect: Proofs.}%

We proceed to our proof of Theorem~\ref{thm: Main inequalities.}.  Let $K\subset\mathbb{C}$ be bounded and convex, with diameter $s\geq0$.  Let $n$ and $k$ be positive integers with $k\leq n$, and let $\epsilon>0$ be given.  Suppose that $f$ is a rational function with $n$ zeros and poles combined (counting multiplicity), and having at least $k$ zeros lying in $K$ (that is, $\#_z(f,\mathbb{C})+\#_p(f,\mathbb{C})=n$, and $\#_z(f,K)\geq k$).  In the case that $k=n$, $f$ is a polynomial with all of its zeros lying in $K$, so all of its critical points lie in $K$ by the classical Gauss--Lucas theorem.  Thus we consider only the case $k<n$.  We first factor $f$ as $f=g\cdot h$, where $g$ is the monic polynomial having for its zeros exactly the at least $k$ zeros of $f$ which lie in $K$, and $h$ is a rational function with at most $n-k$ zeros and poles combined.

Since the locations of the critical points of $g$, $h$, and $gh$ are continuous as a function of the zeros and poles of $g$ and $h$, we may assume without loss of generality that the zeros and poles of $g$ and $h$ are all distinct.  Define the set $$E=\left\{z\in\mathbb{C}:\dfrac{\epsilon}{2}< d(z,K)<\epsilon\right\},$$ where $d(z,K)$ denotes the distance $d(z,K)=\ds\inf_{w\in K}(|z-w|)$.  Since $K$ is convex, $E$ is a bounded conformal annulus (that is, a bounded open subset of $\mathbb{C}$ such that $E^c$ has a single bounded component) that contains $K$ in its bounded face.
Let $\mathcal{A}$ denote the union of the closed balls centered at the zeros and poles of $h$, each with radius $\dfrac{\epsilon}{8(n-k)}$.  $\mathcal{A}$ consists of at most $n-k$ closed balls, each with diameter $\dfrac{\epsilon}{4(n-k)}$, so the diameter of any component of $\mathcal{A}$ is at most $\dfrac{\epsilon}{4}$.  However for any points $z$ in the bounded face of $E$ and $z'$ in the unbounded face of $E$, the triangle inequality implies that $|z-z'|\geq\dfrac{\epsilon}{2}$, so it follows that $E\setminus\mathcal{A}$ is still an open set with $K$ contained in a bounded face of $E\setminus\mathcal{A}$.



Since $K$ is contained in a bounded face of $E\setminus\mathcal{A}$, and $E\setminus\mathcal{A}$ is open, we may find a smooth path $\gamma$ which lies in $E\setminus\mathcal{A}$, and winds once around the face of $E\setminus\mathcal{A}$ which contains $K$.  Our plan is to apply Rouch\'e's theorem to the functions $\dfrac{g'}{g}$ and $\dfrac{h'}{h}$ on the path $\gamma$.  Thus we wish to have $\left|\dfrac{g'}{g}\right|>\left|\dfrac{h'}{h}\right|$ on $\gamma$.  Fix some point $z_0$ in $\gamma$.

We will first find an upper bound for $\left|\dfrac{h'(z_0)}{h(z_0)}\right|$.  Let $Z_h$ and $P_h$ denote the set of zeros of $h$ and poles of $h$ respectively (recall that $h$ was assumed to have all distinct zeros and poles).  Then we have $|Z_h|+|P_h|\leq n-k$.  After a little arithmetic, we have
\begin{equation}\label{eqn: Breakdown of h'/h.}
\dfrac{h'(z_0)}{h(z_0)}=\ds\sum_{x\in Z_h}\dfrac{1}{z_0-x}-\sum_{y\in P_h}\dfrac{1}{z_0-y}.
\end{equation}
Since $z_0$ lies outside of $\mathcal{A}$, for any $w$ in either $Z_h$ or $P_h$, $|z_0-w|>\epsilon/8(n-k)$.  Thus, using the triangle inequality on Equation~\ref{eqn: Breakdown of h'/h.}, we have \begin{equation}\label{eqn: Bound on h'/h.}\left|\dfrac{h'(z_0)}{h(z_0)}\right|<\sum_{x\in Z_h}\dfrac{1}{\epsilon/8(n-k)}+\sum_{y\in P_h}\dfrac{1}{\epsilon/8(n-k)}=\left(|Z_h|+|P_h|\right)\dfrac{8(n-k)}{\epsilon}\leq\dfrac{8(n-k)^2}{\epsilon}.\end{equation}

Let us now turn our attention to a lower bound for $\left|\dfrac{g'(z_0)}{g(z_0)}\right|$.  Let $w_0$ denote the point in $K$ which is closest to $z_0$.  Since $\gamma\subset E$, we have $\epsilon/2<|z_0-w_0|<\epsilon$.  After applying the appropriate distance preserving affine transformation to the plane, we may assume without loss of generality that $z_0=0$ and that $w_0$ lies in the real line, in the interval $(-\epsilon,-\epsilon/2)$.  The choice of $w_0$ and the convexity of $K$ implies that $K$ lies in the left half of the circle centered at $w_0$ with radius $s$.  Let $F$ denote this left half-circle.  Let $Z_g\subset F$ denote the set of zeros of $g$.  Then we have $|Z_g|\geq k$ (recall that $g$ has no poles), and similarly as in Equation~\ref{eqn: Breakdown of h'/h.}, we have
\begin{equation}\label{eqn: Breakdown of g'/g.}
\dfrac{g'(z_0)}{g(z_0)}=\dfrac{g'(0)}{g(0)}=\ds\sum_{x\in Z_g}\dfrac{1}{-x}.
\end{equation}

Now, for any point $z\in \mathbb{C}$, $|z|\geq\Re(z)$, and if $z\in F$, we have $|z|\leq s+\epsilon$ and $\Re(-z)\geq\epsilon/2$, so taking absolute values in Equation~\ref{eqn: Breakdown of g'/g.}, we obtain
\begin{equation}\label{eqn: First estimate on g'/g.}
\left|\dfrac{g'(0)}{g(0)}\right|\geq\Re\left(\sum_{x\in Z_g}\dfrac{1}{-x}\right)=\sum_{x\in Z_g}\Re\left(\dfrac{1}{-x}\right)=\sum_{x\in Z_g}\dfrac{\Re(-x)}{|-x|^2}\geq\sum_{x\in Z_g}\dfrac{\epsilon/2}{(s+\epsilon)^2}\geq\dfrac{k\epsilon}{2(s+\epsilon)^2}.
\end{equation}

Since the choice of $z_0$ was arbitrary, Equations~\ref{eqn: Bound on h'/h.} and~\ref{eqn: First estimate on g'/g.} give us that on $\gamma$,
\begin{equation}\label{eqn: Bounds on gamma.}
\left|\dfrac{h'}{h}\right|<\dfrac{8(n-k)^2}{\epsilon}\text{  and  }\dfrac{k\epsilon}{2(s+\epsilon)^2}<\left|\dfrac{g'}{g}\right|.
\end{equation}

Thus if \begin{equation}\label{eqn: Constants needed.}\dfrac{8(n-k)^2}{\epsilon}<\dfrac{k\epsilon}{2(s+\epsilon)^2},\end{equation} then on $\gamma$, $\left|h'/h\right|<\left|g'/g\right|$.  Let $\Omega$ denote the bounded face of $\gamma$.  According to Rouch\'e's theorem, we could then conclude that
\begin{equation}\label{eqn: Rouche set-up.}
\#_z\left(\dfrac{g'}{g},\Omega\right)-\#_p\left(\dfrac{g'}{g},\Omega\right)=\#_z\left(\dfrac{g'}{g}+\dfrac{h'}{h},\Omega\right)-\#_p\left(\dfrac{g'}{g}+\dfrac{h'}{h},\Omega\right).\end{equation}

Since $g$ is a polynomial with all distinct zeros, and all of the zeros and critical points of $g$ lie in $K$, Lemma~\ref{lem: Lemma.} then gives us
\begin{equation}\label{eqn: Equation g.}
\#_z\left(\dfrac{g'}{g},\Omega\right)-\#_p\left(\dfrac{g'}{g},\Omega\right)=\#_c(g,\Omega)-(\#_z(g,\Omega)+\#_p(g,\Omega)).
\end{equation}
On the other hand, through the magic of the product rule, $$\dfrac{g'}{g}+\dfrac{h'}{h}=\dfrac{(gh)'}{gh},$$ so again by Lemma~\ref{lem: Lemma.}, we would have
\begin{equation}\label{eqn: Equation g+h.}
\#_z\left(\dfrac{g'}{g}+\dfrac{h'}{h},\Omega\right)-\#_p\left(\dfrac{g'}{g}+\dfrac{h'}{h},\Omega\right)=\#_c(gh,\Omega)-(\#_z(gh,\Omega)+\#_p(gh,\Omega)).
\end{equation}

Performing the substitutions indicated by Equations~\ref{eqn: Equation g.} and~\ref{eqn: Equation g+h.}, Equation~\ref{eqn: Rouche set-up.} becomes
\begin{equation}\label{eqn: After substitutions.}
\#_c(g,\Omega)-(\#_z(g,\Omega)+\#_p(g,\Omega))=\#_c(gh,\Omega)-(\#_z(gh,\Omega)+\#_p(gh,\Omega)).
\end{equation}
Solving Equation~\ref{eqn: After substitutions.} for $\#_c\left(gh,\Omega\right)$, and using the fact that $\Omega\subset K_\epsilon$, we obtain
\begin{equation}\label{eqn: Critical points for K_epsilon.}
\#_c(gh,K_\epsilon)\geq\#_c(gh,\Omega)=\#_c(g,K)+(\#_z(gh,\Omega)+\#_p(gh,\Omega))-(\#_z(g,K)+\#_p(g,K)).
\end{equation}
Since $\#_z(gh,\Omega)-\#_z(g,\Omega)=\#_z(h,\Omega)$, and $\#_p(gh,\Omega)-\#_p(g,\Omega)=\#_p(h,\Omega)$, Equation~\ref{eqn: Critical points for K_epsilon.} provides us with our desired inequality:
\begin{equation}\label{eqn: Final inequality for critical points of gh in K_epsilon.}
\#_c(gh,K_\epsilon)\geq\#_c(g,K)+(\#_z(h,\Omega)+\#_p(h,\Omega))\geq\#_c(g,K)\geq k-1.
\end{equation}

Recall that Equation~\ref{eqn: Final inequality for critical points of gh in K_epsilon.} was obtained subject to the assumption that $$\dfrac{8(n-k)^2}{\epsilon}<\dfrac{k\epsilon}{2(s+\epsilon)^2}.$$  This inequality may easily be rearranged into the first inequality found in the statement of Theorem~\ref{thm: Main inequalities.}.  If we view $n$ and $k$ as fixed, and solve this inequality for $\epsilon$, we obtain as a sufficient condition for $\mathcal{RAGL}(n,k,\epsilon)$ the inequality
\begin{equation}\label{eqn: n and k fixed inequality.}
\epsilon>\dfrac{4s(n-k)}{\sqrt{k}-4(n-k)}-1.
\end{equation}
This is the conclusion of Corollary~\ref{cor: n and k held constant, solve for epsilon.}.  Note that in solving for $\epsilon$ to obtain the above inequality, the assumption was made that $\sqrt{k}-4(n-k)>0$.

We now turn our attention to Corollary~\ref{cor: Our bound on psi_1.}, and the special case where $K$ is a disk with diameter $s$.  In that case, our earlier work still holds, but we are able to sharpen our estimate on $\left|\dfrac{g'}{g}\right|$.  We again arrive at the situation where we assume that $z_0\in\gamma$ equals $0$, and that the closest point to $z_0$ in $K$ is $w_0\in(-\epsilon,-\epsilon/2)$ (so that $K$ is the disk with diameter $s$, centered at the real number $w_0-s/2$).  For any $w\in K$, with $x=\Re(w)$ and $y=\Im(w)$, $$\Re\left(\dfrac{1}{-w}\right)=\dfrac{-x}{x^2+y^2}.$$  Using the techniques of undergraduate multi-variable calculus, we find that this function of $x$ and $y$ is minimized on $K$ at the furthest left point of the disk, namely $w_0-s$.  Thus we have that for $w\in K$, $$\Re\left(\dfrac{1}{-w}\right)\geq\dfrac{1}{-w_0+s}\geq\dfrac{1}{\epsilon+s}.$$  Adopting $1/(\epsilon+s)$ as our lower bound for $\Re(1/-w)$ (for each zero $w$ of $g$), Equation~\ref{eqn: First estimate on g'/g.} becomes
\begin{equation}\label{eqn: Inequality for g'(0)/g(0) in disk case.}
\left|\dfrac{g'(0)}{g(0)}\right|\geq\dfrac{k}{\epsilon+s}.
\end{equation}

Recall that this then implies that $$\left|\dfrac{g'}{g}\right|\geq\dfrac{k}{\epsilon+s}$$ on $\gamma$.  Thus the sufficient condition for $\mathcal{RAGL}(n,k,\epsilon)$ which we obtain in this case is $$\dfrac{8(n-k)^2}{\epsilon}<\dfrac{k}{\epsilon+s},$$ which, after some simple arithmetic, provides the second conclusion of Theorem~\ref{thm: Main inequalities.}.  If we view $n$ and $k$ as fixed, and solve this inequality for $\epsilon$ (this arithmetic requires the assumption that $k>8(n-k)^2$), we obtain the inequality $$\epsilon>\dfrac{8s(n-k)^2}{k-8(n-k)^2}.$$

From this inequality, we derive the result of Corollary~\ref{cor: Our bound on psi_1.} that if $k>8(n-k)^2$, then $$\psi_1(n,k)\leq\dfrac{8(n-k)^2}{k-8(n-k)^2}.$$

\section{AN ADDITIONAL CONJECTURE}\label{sect: An additional conjecture.}

The hueristic which underlies our proof of Theorem~\ref{thm: Main inequalities.} is that when our rational function $f$ is factored as $f=g\cdot h$ (where $g$ gets all the zeros of $f$ lying in $K$, and $h$ gets all remaining zeros and poles of $f$), then the relatively few zeros and poles of $h$ cannot drag the critical points of $g$ very far outside of $K$.  This hueristic seems to accommodate perfectly well the possibility of $K$ being unbounded (this possibility is present in the classical Gauss--Lucas theorem as well), and of $g$ being a rational function (now possessing all of the zeros \textbf{and} poles of $f$ which lie in $K$).  Of course, if $g$ is a rational function, its critical points need not in general lie in $K$, so that must be added as an assumption.  We therefore extend Conjecture~\ref{conj: epsilon fixed goal.} (with the function $f$ already factored as $f=g\cdot h$) as follows.

\begin{conjecture}\label{conj: Allow unbounded and poles.}
Let $K\subset\mathbb{C}$ be convex, and let $\epsilon>0$ be given.  There exists some constant $C>0$ for which the following holds.  Let $g,h\in\mathbb{C}(z)$ be rational functions, and assume that all zeros, poles, and critical points of $g$ lie in $K$.  If $\dfrac{\#_z(g,\mathbb{C})+\#_p(g,\mathbb{C})}{\#_z(h,\mathbb{C})+\#_p(h,\mathbb{C})}>C$, then $\#_c(gh,K_\epsilon)\geq\#_c(g,K)$.
\end{conjecture}

\bibliographystyle{plain}
\bibliography{APPROXGUASSLUCAS}

\begin{thebibliography}{1}

\bibitem{B}
M.~Biernaki.
\newblock Sur les z\'eros des polyno\^omes et sur les fonctions enti\`eres dont
  le development taylorien presente des lacunes.
\newblock {\em Bull. Soc. Math. France}, 69(2):197--203, 1945.

\bibitem{BKS}
R.~Boegvad, D.~Khavinson, and B.~Shapiro.
\newblock On asymptotic {Gauss}--{Lucas} theorem.
\newblock {\em arXiv:1510.02339}, 2015.

\bibitem{K}
S.~Kakeya.
\newblock On zeros of a polynomial and its derivatives.
\newblock {\em T\^ohoku Math. J.}, 11:5--16, 1917.

\bibitem{M}
M.~Marden.
\newblock Kakeya's problem on the zeros of the derivative of a polynomial.
\newblock {\em Trans. Amer. Math. Soc.}, 45:355--368, 1939.

\bibitem{M2}
M.~Marden.
\newblock {\em The Geometry of the Zeros}.
\newblock AMS, 1949.

\bibitem{T}
V.~Totik.
\newblock The {Gauss}--{Lucas} theorem in an asymptotic sense.
\newblock {\em Bul. London Math. Soc.}, 48(5):848--854, 2016.

\end{thebibliography}

\end{document}